\documentclass[a4paper, 12pt]{amsproc}

\usepackage{fullpage}
\usepackage{amsmath, amsthm, amssymb, mathtools, mathrsfs}

\usepackage{hyperref}

\usepackage{chessfss}
\usepackage{tikz}
\usetikzlibrary{intersections, calc, arrows.meta}

\usepackage{graphicx}
\usepackage{here}
\usepackage{time}

\usepackage{xcolor}
\usepackage[capitalize,nameinlink,noabbrev,nosort]{cleveref}
\hypersetup{
	colorlinks=true,       
	linkcolor=brown,          
	citecolor=brown,        
	filecolor=brown,      
	urlcolor=brown,           
}

\makeatletter
\@namedef{subjclassname@2020}{%
  \textup{2020} Mathematics Subject Classification}
\makeatother

\newtheorem{theoremcounter}{Theorem Counter}[section]

\theoremstyle{definition}
\newtheorem{definition}[theoremcounter]{Definition}
\newtheorem{remark}[theoremcounter]{Remark}
\newtheorem{example}[theoremcounter]{Example}

\theoremstyle{plain}

\newtheorem{corollary}[theoremcounter]{Corollary}

\newtheorem{theorem}[theoremcounter]{Theorem}

\numberwithin{equation}{section}

\newcommand{\Z}{\mathbb{Z}}

\newcommand{\scB}{\mathscr{B}}

\newcommand{\pmat}[1]{\begin{pmatrix}#1\end{pmatrix}}

\newcommand{\smat}[1]{\bigl(\begin{smallmatrix}#1\end{smallmatrix}\bigr)}

\def\piros#1{{\color{red}#1}}%
\def\kek#1{{\color{blue}#1}}%

\def\sts#1#2{\genfrac{\{}{\}}{0pt}{}{#1}{#2}}

\newcommand{\wclr}{w_{\mathcal{C}}^{\mathrm{lr}}}
\newcommand{\wcrb}{w_{\mathcal{C}}^{\mathrm{br}}}
\newcommand{\wcRL}{w_{\mathcal{C}}^{\mathrm{RL}}}
\newcommand{\wtst}{w_{\mathcal{T}}^{\mathrm{st}}}
\newcommand{\wtl}{w_{\mathcal{T}}^{\leftarrow}}
\newcommand{\wttl}{w_{\widetilde{\mathcal{T}}}^{\leftarrow}}
\newcommand{\wttd}{w_{\widetilde{\mathcal{T}}}^{\downarrow}}
\newcommand{\wtc}{w_{\mathscr{T}}^{\mathrm{ch}}}
\newcommand{\welr}{w_{\mathcal{E}}^{\mathrm{lr}}}

\begin{document}

\title[]{Bijective enumerations for symmetrized poly-Bernoulli polynomials} 

\author{Minoru Hirose} 
\address{Institute for Advanced Research, Nagoya University,
Furo-cho, Chikusa-ku, Nagoya 464-8601, Japan}
\email{minoru.hirose@math.nagoya-u.ac.jp}

\author{Toshiki Matsusaka}
\address{Institute for Advanced Research, Nagoya University,
Furo-cho, Chikusa-ku, Nagoya 464-8601, Japan}
\email{matsusaka.toshiki@math.nagoya-u.ac.jp} 

\author{Ryutaro Sekigawa} 
\address{Graduate School of Science and Technology, Tokyo University of Science,
Noda, Chiba 278-8510, Japan}
\email{sekigawa.r@gmail.com} 

\author{Hyuga Yoshizaki} 
\address{Graduate School of Science and Technology, Tokyo University of Science,
Noda, Chiba 278-8510, Japan}
\email{yoshizaki\_hyuuga@ma.noda.tus.ac.jp} 

\thanks{The first author was supported by JSPS KAKENHI Grant Number JP18K13392, and the second author was supported by JSPS KAKENHI Grant Number JP20K14292. 
}



\maketitle

\begin{abstract}
	Recently, B\'{e}nyi and the second author introduced two combinatorial interpretations for symmetrized poly-Bernoulli polynomials. In the present study, we construct bijections between these combinatorial objects. We also define various combinatorial polynomials and prove that all of these polynomials coincide with symmetrized poly-Bernoulli polynomials.
\end{abstract}

\section{Introduction}

For non-negative integers $n, k \geq 0$, the (normalized) \emph{symmetrized poly-Bernoulli polynomial} $\widehat{\scB}_n^k(x)$ is defined by
\begin{align}\label{SPB-exp}
	\widehat{\scB}_n^k(x) = \sum_{j=0}^{\min(n,k)} j! (x+1)^{\overline{j}} \sts{n+1}{j+1} \sts{k+1}{j+1} \in \Z[x].
\end{align}
Here, $\sts{\cdot}{\cdot}$ is the Stirling number of the second kind (see~\cite[Section 2.1]{ArakawaIbukiyamaKaneko2014}), and $(x+1)^{\overline{j}} = (x+1)(x+2) \cdots (x+j)$ is the rising factorial. The prototype, i.e., the symmetrized poly-Bernoulli number, was introduced by Kaneko, Sakurai, and Tsumura~\cite{KanekoSakuraiTsumura2018} to generalize the duality of poly-Bernoulli numbers with negative indices. 

The poly-Bernoulli number was introduced by Kaneko~\cite{Kaneko1997} in 1997. Since Brewbaker~\cite{Brewbaker2008} and Launois~\cite{Launois2005} pointed out that poly-Bernoulli numbers appear in enumeration problems, this topic has been studied from a combinatorial viewpoint (see~\cite{BenyiHajnal2015}). By definition, the coefficients of symmetrized poly-Bernoulli polynomials are non-negative integers. Therefore, it is natural to ask about their combinatorial meanings.

Recently, B\'{e}nyi--Matsusaka~\cite{BenyiMatsusaka2021} introduced two combinatorial objects to answer this question. Inspired by this research, we provide further combinatorial aspects for symmetrized poly-Bernoulli polynomials. Furthermore, we answer some problems left unsolved by B\'{e}nyi--Matsusaka. The remainder of the present article is organized as follows.

First, we recall the two combinatorial objects, (barred) Callan sequences $\mathcal{C}_n^k$ and alternative tableaux $\mathcal{T}_n^k$ in \cref{s2}. Both of these objects define the same polynomial $\widehat{\mathscr{B}}_n^k(x)$. However, the proof is indirect due to using recurrence relations, and the authors~\cite{BenyiMatsusaka2021} failed to find a direct bijection between these two objects. We therefore present two types of combinatorial bijections.

More precisely, we introduce the following. 

\begin{definition}
	For a pair $(\mathcal{P}, w)$ of a finite set of combinatorial objects and a weight function $w: \mathcal{P} \to \mathbb{Z}_{\geq 0}$, we define the polynomial
\[
	\mathcal{P}(x) = \mathcal{P}(x; w) = \sum_{\lambda \in \mathcal{P}} x^{w(\lambda)}.
\]
Let $(\mathcal{P}_1, w_1)$ and $(\mathcal{P}_2, w_2)$ be two such pairs. A function $f : \mathcal{P}_1 \to \mathcal{P}_2$ is called a \emph{bijection between $(\mathcal{P}_1, w_1)$ and $(\mathcal{P}_2, w_2)$} if $f$ is bijective and satisfies $w_2(f(\lambda)) = w_1(\lambda)$ for any $\lambda \in \mathcal{P}_1$.
\end{definition}

If there exists a bijection between $(\mathcal{P}_1, w_1)$ and $(\mathcal{P}_2, w_2)$, then we obtain the equation $\mathcal{P}_1(x; w_1) = \mathcal{P}_2(x; w_2)$. In \cref{s2-2}, we construct a direct bijection $(\mathcal{C}_n^k, \wclr) \to (\mathcal{T}_n^k, \wtst)$, where $\wclr$ and $\wtst$ are weight functions introduced in~\cite{BenyiMatsusaka2021}. In \cref{s3}, we give another bijection $(\mathcal{T}_n^k, \wtl) \to (\mathcal{C}_n^k, \wclr)$ via a sequence of bijections.

Throughout this article, we provide various combinatorial objects and weights. The following table lists the models considered herein.

\begin{table}[H]
\centering
\begin{tabular}{c|l}
	Section & $(\mathcal{P}, w)$ \\
	\hline
	\cref{s2} & $(\mathcal{C}_n^k, \wclr)$, $(\mathcal{T}_n^k, \wtst)$\\
	\cref{s3-1} & $(\mathcal{T}_n^k, \wtl)$, $(\widetilde{\mathcal{T}}_n^k, \wttl)$\\
	\cref{s3-2} & $(\mathscr{T}_n^k, \wtc)$, $(\mathcal{C}_n^k, \wcrb)$\\
	\cref{s4-1} & $(\mathcal{E}_n^k, \welr)$\\
	\cref{s4-2} & $(\mathcal{C}_n^k, \wcRL)$\\
	\cref{s4-3} & $(\widetilde{\mathcal{T}}_n^k, \wttd)$
\end{tabular}
\end{table}

\begin{theorem}\label{main-theorem} 
	The polynomials $\mathcal{P}(x; w)$ defined from the above nine combinatorial models all coincide with the symmetrized poly-Bernoulli polynomial $\widehat{\mathscr{B}}_n^k(x)$.
\end{theorem}

Although \cref{SPB-exp} immediately implies the duality $\widehat{\mathscr{B}}_n^k(x) = \widehat{\mathscr{B}}_k^n(x)$, this is unclear from the combinatorial models of B\'{e}nyi--Matsusaka. As an application of our results, we explain the duality combinatorially in \cref{s4-3}.

\section{A bijection between two combinatorial models} \label{s2}

B\'{e}nyi and the second author~\cite{BenyiMatsusaka2021} introduced combinatorial models for symmetrized poly-Bernoulli polynomial $\widehat{\scB}_n^k(x)$. In this section, we first recall these models, and then provide a bijection between these models. 

\subsection{Double Callan permutations} \label{s2-1}

Throughout this article, let $n$ and $k$ be non-negative integers. 

\begin{definition}
	A \emph{double Callan permutation} of size $n \times k$ is a pair of (possible empty) strings $S_1 = a_1 \cdots a_r$ and $S_2 = b_1 \cdots b_s$ with $r+s = n+k$ such that
	\begin{enumerate}
		\item the terms satisfy $\{a_1, \dotsm a_r\} \sqcup \{b_1, \dots, b_s\} = \{\piros{1}, \dots, \piros{n}, \kek{1}, \dots, \kek{k}\}$ with $r, s \geq 0$,
		\item $a_1$ is blue, and $b_1$ is red, and
		\item all substrings of elements consisting of the same color are in decreasing order.
	\end{enumerate}
	We let $\mathcal{C}_n^k$ denote the set of all double Callan permutations of size $n \times k$.
\end{definition}

\begin{example}
	The following is an example of double Callan permutations of size $7 \times 6$:
	\[
		S_1 = \kek{6} \piros{76} \kek{51} \piros{31} \kek{3} \piros{2} \quad \text{and} \quad S_2 = \piros{4} \kek{4} \piros{5} \kek{2}.
	\]
\end{example}

The double Callan permutations are essentially the same as the barred Callan sequences studied in~\cite{BenyiMatsusaka2021}. Indeed, we can express a pair of strings as
\begin{align}\label{S1S2-exp}
	S_1 = B_1 R_1 \cdots B_\ell R_\ell B_{\ell+1} \quad \text{and} \quad S_2 = R'_1 B'_1 \cdots R'_m B'_m R'_{m+1},
\end{align}
where $R$ and $B$ are substrings consisting of red and blue elements, respectively. Here, $\ell, m \geq 0$ and the substrings $B_{\ell+1}$ and $R'_{m+1}$ could be empty. This expression defines a barred Callan sequence $(B_1; R_1) \cdots (B_\ell; R_\ell) | (B'_m; R'_m) \cdots (B'_1; R'_1) (B_{\ell+1}, \kek{*}; R'_{m+1}, \piros{*})$. Therefore, by reusing the terminology for barred Callan sequences, we refer to $B_{\ell+1}, R'_{m+1}$ as \emph{extra blocks} and refer to other substrings $R, B$ as \emph{ordinary blocks}. Moreover, we refer to a pair of red and blue blocks having the same sub-(super-)script as a \emph{Callan pair}. Note that if $\ell = 0$ (resp. $m = 0$), then the block $B_1$ (resp. $R'_1$) is the extra block.

\begin{definition}\label{ordinary}
For each double Callan permutation $\lambda = (S_1, S_2) \in \mathcal{C}_n^k$ given as in \cref{S1S2-exp}, we define its weight $\wclr (\lambda) \in \Z_{\geq 0}$ using the left-to-right minimum as follows:
\begin{enumerate}
	\item Consider the minimum of each blue substring $B_1, \dots, B_\ell$ in $S_1$ to obtain a sequence $\pi = \pi_1 \pi_2 \cdots \pi_\ell$. Here, we ignore the last blue substring $B_{\ell+1}$.
	\item Count the number of $1 \leq i \leq \ell$ such that if $j < i$ then $\pi_i < \pi_j$.
\end{enumerate}
\end{definition}

For the above example, we obtain the sequence $\pi = \kek{613}$. Then, the weight is given by $\wclr(\lambda) = 2$.

\begin{definition}\label{def-Callan-polynomial}
	For any $n, k \geq 0$, we define the \emph{Callan polynomial} by
	\[
		\mathcal{C}_n^k(x) = \mathcal{C}_n^k(x; \wclr) = \sum_{\lambda \in \mathcal{C}_n^k} x^{\wclr(\lambda)}.
	\]
\end{definition}

The following explicit properties of the Callan polynomials are known.

\begin{theorem} \cite{BenyiMatsusaka2021} \label{Callan-explicit}
	The Callan polynomials satisfy the explicit formula
	\begin{align}\label{Cnk-explicit}
		\mathcal{C}_n^k(x) = \sum_{j=0}^{\min(n,k)} j! (x+1)^{\overline{j}} \sts{n+1}{j+1} \sts{k+1}{j+1},
	\end{align}
	and the generating function
	\begin{align*}
		\sum_{n=0}^\infty \sum_{k=0}^\infty \mathcal{C}_n^k(x) \frac{X^n}{n!} \frac{Y^k}{k!} = \frac{e^{X+Y}}{(e^X + e^Y - e^{X+Y})^{x+1}}.
	\end{align*}
	In particular, $\mathcal{C}_n^k(x) = \widehat{\scB}_n^k(x)$. 
\end{theorem}

\begin{remark}\label{Stirling2}
	For positive integers $n$ and $k$, the Stirling number of the second kind $\sts{n}{k}$ counts the number of ways to divide a set of $n$ elements into $k$ nonempty sets. The Stirling numbers satisfy the recurrence formula
	\[
		\sts{n+1}{k} = \sts{n}{k-1} + k \sts{n}{k}
	\]
	with the initial values $\sts{0}{0} = 1$ and $\sts{n}{0} = \sts{0}{k} = 0$ ($n, k \neq 0$). From this definition, the above explicit formula \cref{Cnk-explicit} immediately follows (see~\cite[Section 3]{BenyiMatsusaka2021}).
\end{remark}

\subsection{Alternative tableaux} \label{s2-3}

The second combinatorial model for $\widehat{\scB}_n^k(x)$ is given by alternative tableaux of rectangular shape. An alternative tableau of general shape was introduced by Viennot~\cite{Viennot2008} and studied by Nadeau~\cite{Nadeau2011}. Here, we recall its definition and the weight function $\wtst: \mathcal{T}_n^k \to \Z_{\geq 0}$ introduced in~\cite{BenyiMatsusaka2021}.

\begin{definition}\label{AltTab-def}
	Let $n, k$ be positive integers. An \emph{alternative tableau} of rectangular shape of size $n \times k$ is a rectangle with a partial filling of the cells with left arrows $\leftarrow$ and down arrows $\downarrow$, such that all cells pointed by an arrow are empty. We let $\mathcal{T}_n^k$ denote the set of all alternative tableaux with a rectangular shape and a size of $n \times k$. 
	
	For each $\lambda \in \mathcal{T}_n^k$,
\begin{enumerate}
	\item consider the first from the top consecutive rows that contain left arrows $\leftarrow$, and
	\item count the number of left arrows $\leftarrow$ such that all $\leftarrow$ in the upper rows are located further to the right.
\end{enumerate}
We let $\wtst(\lambda)$ denote the number of such left arrows, (the superscript ``$\mathrm{st}$" of which is an abbreviation of ``stair").
\end{definition}

\begin{example}\label{AltTab76}
	The following alternative tableau $\lambda \in \mathcal{T}_7^6$ has a weight $\wtst(\lambda) = 2$.
	
	\begin{figure}[H]
	\begin{tikzpicture}
		\draw (0,0)--(3,0); \draw (0,1/2)--(3,1/2); \draw (0,2/2)--(3,2/2); \draw (0,3/2)--(3,3/2); \draw (0,4/2)--(3,4/2); \draw (0,5/2)--(3,5/2); \draw (0,6/2)--(3,6/2); \draw (0,7/2)--(3,7/2);
		\draw (0,0)--(0,7/2); \draw (1/2,0)--(1/2,7/2); \draw (2/2,0)--(2/2,7/2); \draw (3/2,0)--(3/2,7/2); \draw (4/2,0)--(4/2,7/2); \draw (5/2,0)--(5/2,7/2); \draw (6/2,0)--(6/2,7/2);
		\draw (6/2-1/4,7/2-1/4) node{$\piros{\leftarrow}$}; \draw (6/2-1/4,6/2-1/4) node{$\leftarrow$}; \draw (4/2-1/4,5/2-1/4) node{$\piros{\leftarrow}$}; \draw (4/2-1/4,4/2-1/4) node{$\downarrow$}; \draw (3/2-1/4, 2/2-1/4) node{$\leftarrow$}; \draw (3/2-1/4, 1/2-1/4) node{$\downarrow$}; \draw (2/2-1/4, 4/2-1/4) node{$\downarrow$}; \draw (1/2-1/4, 3/2-1/4) node{$\leftarrow$}; \draw (1/2-1/4, 1/2-1/4) node{$\leftarrow$}; 
		
		\draw[line width=2pt] (6/2,7/2)--(5/2,7/2)--(5/2,5/2)--(3/2,5/2)--(3/2,4/2);
	\end{tikzpicture}
	\caption{An alternative tableau of size $7 \times 6$ with an indication of its weight.}
	\label{AltTab76}
	\end{figure}
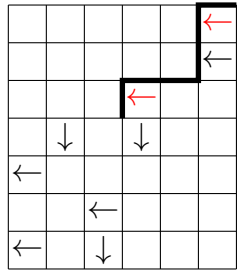
\end{example}

In a similar manner as above, we define the polynomial $\mathcal{T}_n^k(x)$ as
\[
	\mathcal{T}_n^k(x) = \mathcal{T}_n^k(x; \wtst) = \sum_{\lambda \in \mathcal{T}_n^k} x^{\wtst(\lambda)}.
\]

\begin{theorem}\cite{BenyiMatsusaka2021} \label{T=C}
	Let $\mathcal{T}_n^0(x) = \mathcal{T}_0^k(x) = 1$. For any integers $n,k \geq 0$, the polynomial $\mathcal{T}_n^k(x)$ coincides with $\mathcal{C}_n^k(x)$, i.e., $\mathcal{T}_n^k(x) = \widehat{\scB}_n^k(x)$.
\end{theorem}

This theorem was proven by showing that both polynomials $\mathcal{C}_n^k(x)$ and $\mathcal{T}_n^k(x)$ satisfy the same recursion
\begin{align}\label{recurrence}
	\widehat{\scB}_n^k(x) = (n+1) \widehat{\scB}_n^{k-1}(x) + x \sum_{j=0}^{n-1} {n \choose j} \widehat{\scB}_{j}^{k-1}(x) + \sum_{j=1}^{n-1} {n \choose j-1} \widehat{\scB}_{j}^{k-1}(x).
\end{align}

\subsection{A combinatorial bijection} \label{s2-2}

In this subsection, we construct a bijection between these two models.

Both combinatorial models have their own advantages and disadvantages. On one hand, although it is quite easy to show that the polynomial $\mathcal{T}_n^k(x)$ satisfies the recursion in \cref{recurrence}, it is difficult to check the explicit formula (\cref{Callan-explicit}) for $\mathcal{T}_n^k(x)$. On the other hand, as we mentioned in \cref{Stirling2}, the explicit formula immediately follows from the definition of $\mathcal{C}_n^k(x)$ by simply enumerating the objects. However, it is slightly complicated to show that the Callan polynomials $\mathcal{C}_n^k(x)$ satisfy the recursion. 

Indeed, the authors~\cite[Theorem 14]{BenyiMatsusaka2021} used the following auxiliary map $\varphi$ to show the recursion for the Callan polynomials. Here, we recall this map.

\begin{definition}\label{def-auxiliary}
	For any integers $n \geq 0$ and $k > 0$, we define a map $\varphi: \mathcal{C}_n^k \to \mathcal{C}_{\leq n}^{k-1} := \bigcup_{i=0}^n \mathcal{C}_i^{k-1}$ as follows. Let $(S_1, S_2) \in \mathcal{C}_n^k$ be expressed as in \cref{S1S2-exp}.
	\begin{itemize}
		\item[(1)] If $\kek{k}$ is in the extra block $B_{\ell+1}$, then remove $\kek{k}$.
		\item[(2)] If $\kek{k}$ is alone in the first ordinary block $B_1$, then remove the first Callan pair $B_1 R_1$.
		\item[$(3)$] Otherwise, let $R$ be the ordinary red block that forms a Callan pair with the blue block containing $\kek{k}$. Then, remove $\kek{k}$ and replace $R$ with the red element $\piros{0}$. Finally, rearrange the position of $\piros{0}$ in decreasing order if needed.
	\end{itemize}
	After that, we rearrange red elements from $1$.
\end{definition}

\begin{example}\label{varphi-ex}
	For a double Callan permutation $(S_1, S_2) = (\kek{6}\piros{21} \kek{4}\piros{4} \kek{5}, \piros{76}\kek{32} \piros{53}\kek{1}) \in \mathcal{C}_7^6$, we have
	\begin{align*}
		\varphi &: (\kek{6}\piros{21} \kek{4}\piros{4} \kek{5}, \piros{76}\kek{32} \piros{53}\kek{1}) \mapsto (\kek{4}\piros{4} \kek{5}, \piros{76}\kek{32} \piros{53}\kek{1}) \mapsto (\kek{4}\piros{2} \kek{5}, \piros{54}\kek{32} \piros{31}\kek{1}), & &\cdots (2)\\
		\varphi &: (\kek{4}\piros{2} \kek{5}, \piros{54}\kek{32} \piros{31}\kek{1}) \mapsto (\kek{4}\piros{2}, \piros{54}\kek{32} \piros{31}\kek{1}), & &\cdots (1)\\
		\varphi &: (\kek{4}\piros{2}, \piros{54}\kek{32} \piros{31}\kek{1}) \mapsto (\emptyset, \piros{54}\kek{32} \piros{31}\kek{1}) \mapsto (\emptyset,\piros{43}\kek{32} \piros{21}\kek{1}), & &\cdots (2)\\
		\varphi &: (\emptyset, \piros{43}\kek{32} \piros{21}\kek{1}) \mapsto (\emptyset, \piros{0}\kek{2} \piros{21}\kek{1}) \mapsto (\emptyset, \piros{1}\kek{2} \piros{32}\kek{1}), & &\cdots (3)\\
		\varphi &: (\emptyset, \piros{1}\kek{2} \piros{32}\kek{1}) \mapsto (\emptyset, \piros{0}\piros{32}\kek{1}) \mapsto (\emptyset, \piros{320}\kek{1}) \mapsto (\emptyset, \piros{321}\kek{1}), & &\cdots (3)\\
		\varphi &: (\emptyset, \piros{321}\kek{1}) \mapsto (\emptyset, \piros{0}) \mapsto (\emptyset, \piros{1}). & &\cdots (3)
	\end{align*}
\end{example}

As mentioned above, the polynomials $\mathcal{C}_n^k(x)$ and $\mathcal{T}_n^k(x)$ satisfy the same recursion \cref{recurrence}. This tells us that two models $(\mathcal{C}_n^k, \wclr)$ and $(\mathcal{T}_n^k, \wtl)$ have the same recursive structure. Using the map $\varphi$, we construct a bijection from $\mathcal{C}_n^k$ to $\mathcal{T}_n^k$ in a stepwise manner on $k$. We first define the map from $\mathcal{C}_n^k$ to $\mathcal{T}_n^1$.

\begin{definition}\label{BtoT}
	For a given double Callan permutation $\lambda \in \mathcal{C}_n^k$, we create an alternative tableau $\lambda_k \in \mathcal{T}_n^1$ by following the steps below. If $\kek{k}$ is not in the extra block, let $R$ be as described in \cref{def-auxiliary} (3).
	\begin{enumerate}
		\item If $\kek{k}$ is in the extra block, then $\lambda_k = \emptyset$.
		\item If $\kek{k}$ is alone in the first ordinary block $B_1$, then the (1,1)-entry is $\leftarrow$. Moreover, 
			\begin{enumerate}
				\item if $\piros{1} \in R$, then the $(\ell, 1)$-entry is $\leftarrow$ for $\ell \in R$, and
				\item if $\piros{1} \not\in R$, then the $(m, 1)$-entry is $\downarrow$ for $m = \max R$, and the $(\ell, 1)$-entry is $\leftarrow$ for $\ell \in R \setminus \{m\}$.
			\end{enumerate}
		\item Otherwise,
		\begin{enumerate}
			\item if $|R| = 1$, then the $(\ell,1)$-entry is $\downarrow$ for $\ell \in R$,
			\item if $|R| > 1$ and $\piros{1} \in R$, then the $(\ell, 1)$-entry is $\leftarrow$ for $\ell \in R \setminus \{1\}$, and
			\item if $|R| > 1$ and $\piros{1} \not\in R$, then the $(m,1)$-entry is $\downarrow$ for $m = \max R$, and $(\ell,1)$-entry is $\leftarrow$ for $\ell \in R \setminus \{m\}$.
		\end{enumerate}
	\end{enumerate}
\end{definition}

We next define the desired map $\mathcal{C}_n^k \to \mathcal{T}_n^k$ inductively. By \cref{BtoT}, we have $\lambda_k \in \mathcal{T}_n^1$. If $\lambda_k$ contains $\ell$ left arrows, then $\varphi(\lambda) \in \mathcal{C}_{n-\ell}^{k-1}$. By applying the map in \cref{BtoT} again to $\varphi(\lambda) \in \mathcal{C}_{n-\ell}^{k-1}$, we obtain $\lambda_{k-1} \in \mathcal{T}_{n - \ell}^1$. By repeating the steps, we have a sequence $\lambda_1 \lambda_2 \cdots \lambda_k$. The concatenation gives an alternative tableau in $\mathcal{T}_n^k$ with the same weight as $\wclr(\lambda)$. 

The bijectiveness can be checked inductively. We sketch the idea of the proof. For any $n$ and $k = 1$, the map $\mathcal{C}_n^1 \to \mathcal{T}_n^1$ defined in \cref{BtoT} is a bijection preserving the weight. We let $\mathcal{T}_{n,\ell}^1$ denote the subset of $\mathcal{T}_n^1$ such that $\lambda \in \mathcal{T}_n^1$ contains $\ell$ left arrows. Then, our maps induce a bijection $\mathcal{C}_n^k \to \bigcup_{\ell=0}^n (\mathcal{C}_{n-\ell}^{k-1} \times \mathcal{T}_{n,\ell}^1) : \lambda \mapsto (\varphi(\lambda), \lambda_k)$. By the inductive assumption, $\mathcal{C}_{n-\ell}^{k-1}$ is isomorphic to $\mathcal{T}_{n-\ell}^{k-1}$. Thus, $\mathcal{C}_n^k \to \mathcal{T}_n^k$ is bijective. Note that Case (2) in \cref{def-auxiliary} and Case (2) in \cref{BtoT} affect the weight. In particular, we can check that the bijective map preserves the weight.

In conclusion, we have the following:

\begin{theorem}\label{direct-bijection-CT}
	The above map gives a bijection $(\mathcal{C}_n^k, \wclr) \to (\mathcal{T}_n^k, \wtst)$.
\end{theorem}

\begin{example}\label{ExaExa}
	For a double Callan permutation $\lambda = (\kek{6}\piros{21} \kek{4}\piros{4} \kek{5}, \piros{76}\kek{32} \piros{53}\kek{1}) \in \mathcal{C}_7^6$, we already computed the images under $\varphi$ in \cref{varphi-ex}. The corresponding sequence $\lambda_1 \lambda_2 \cdots \lambda_6$ of the alternative tableaux and their concatenation are as follows. This is the alternative tableau given in \cref{AltTab76}.

\begin{figure}[H]
	
	\begin{tikzpicture}
		\draw (0,2)--(1/2,2)--(1/2,7/2)--(0,7/2)--(0,2); \draw (0,5/2)--(1/2,5/2); \draw (0,3)--(1/2,3); \draw (1/4,2+1/4) node{$\leftarrow$}; \draw (1/4,5/2+1/4) node{$\leftarrow$}; \draw (1/4, -1/4) node {$\lambda_1$}; \draw (1/4, -3/4) node {$\vdots$}; \draw (1/4, -3/2) node {$(3)_b$};
		\draw (1,2)--(3/2,2)--(3/2,7/2)--(1,7/2)--(1,2); \draw (1,5/2)--(3/2,5/2); \draw (1,3)--(3/2,3); \draw (1+1/4,3+1/4) node {$\downarrow$}; \draw (1+1/4, -1/4) node {$\lambda_2$}; \draw (1+1/4, -3/4) node {$\vdots$}; \draw (1+1/4, -3/2) node {$(3)_a$};
		\draw (2,3/2)--(5/2,3/2)--(5/2,7/2)--(2,7/2)--(2,3/2); \draw (2,2)--(5/2,2); \draw (2,5/2)--(5/2,5/2); \draw (2,3)--(5/2,3); \draw (2+1/4,3/2+1/4) node{$\downarrow$}; \draw (2+1/4,2+1/4) node{$\leftarrow$}; \draw (2+1/4, -1/4) node {$\lambda_3$}; \draw (2+1/4, -3/4) node {$\vdots$}; \draw (2+1/4, -3/2) node {$(3)_c$};
		\draw (3,1)--(7/2,1)--(7/2,7/2)--(3,7/2)--(3,1); \draw (3,3/2)--(7/2,3/2); \draw (3,2)--(7/2,2); \draw (3,5/2)--(7/2,5/2); \draw (3,3)--(7/2,3); \draw (3+1/4, 5/2+1/4) node {$\downarrow$}; \draw (3+1/4, 3+1/4) node {$\leftarrow$}; \draw (3+1/4, -1/4) node {$\lambda_4$}; \draw (3+1/4, -3/4) node {$\vdots$}; \draw (3+1/4, -3/2) node {$(2)_b$};
		\draw (4,1)--(9/2,1)--(9/2,7/2)--(4,7/2)--(4,1); \draw (4,3/2)--(9/2,3/2); \draw (4,2)--(9/2,2); \draw (4,5/2)--(9/2,5/2); \draw (4,3)--(9/2,3); \draw (4+1/4, -1/4) node {$\lambda_5$}; \draw (4+1/4, -3/4) node {$\vdots$}; \draw (4+1/4, -3/2) node {$(1)$};
		\draw (5,0)--(11/2,0)--(11/2,7/2)--(5,7/2)--(5,0); \draw (5,1/2)--(11/2,1/2); \draw (5,1)--(11/2,1); \draw (5,3/2)--(11/2,3/2); \draw (5,2)--(11/2,2); \draw (5,5/2)--(11/2,5/2); \draw (5,3)--(11/2,3); \draw (5+1/4, 5/2+1/4) node {$\leftarrow$}; \draw (5+1/4, 3+1/4) node {$\leftarrow$}; \draw (5+1/4, -1/4) node {$\lambda_6$}; \draw (5+1/4, -3/4) node {$\vdots$}; \draw (5+1/4, -3/2) node {$(2)_a$};
		
		\draw (7,0)--(10,0); \draw (7,1/2)--(10,1/2); \draw (7,2/2)--(10,2/2); \draw (7,3/2)--(10,3/2); \draw (7,4/2)--(10,4/2); \draw (7,5/2)--(10,5/2); \draw (7,6/2)--(10,6/2); \draw (7,7/2)--(10,7/2);
		\draw (7,0)--(7,7/2); \draw (15/2,0)--(15/2,7/2); \draw (16/2,0)--(16/2,7/2); \draw (17/2,0)--(17/2,7/2); \draw (18/2,0)--(18/2,7/2); \draw (19/2,0)--(19/2,7/2); \draw (20/2,0)--(20/2,7/2);
		\draw (20/2-1/4,7/2-1/4) node{$\leftarrow$}; \draw (20/2-1/4,6/2-1/4) node{$\leftarrow$}; \draw (18/2-1/4,5/2-1/4) node{$\leftarrow$}; \draw (18/2-1/4,4/2-1/4) node{$\downarrow$}; \draw (17/2-1/4, 2/2-1/4) node{$\leftarrow$}; \draw (17/2-1/4, 1/2-1/4) node{$\downarrow$}; \draw (16/2-1/4, 4/2-1/4) node{$\downarrow$}; \draw (15/2-1/4, 3/2-1/4) node{$\leftarrow$}; \draw (15/2-1/4, 1/2-1/4) node{$\leftarrow$}; 
		\draw [dashed] (7,3+1/4)--(19/2,3+1/4); \draw [dashed] (7,5/2+1/4)--(19/2,5/2+1/4); \draw [dashed] (7,2+1/4)--(17/2,2+1/4); \draw [dashed] (7,1/2+1/4)--(8,1/2+1/4);
		
		\draw(6+1/4,3/2)node{$\rightarrow$};
	\end{tikzpicture}
	\caption{The sequence $\lambda_1 \cdots \lambda_6$ and their concatenation.}
\end{figure}
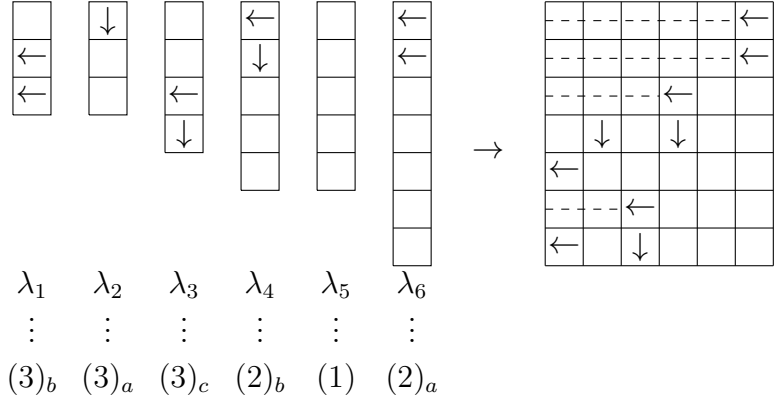
\end{example}

\section{A sequence of bijections}\label{s3}

In the previous section, we studied two combinatorial polynomials $\mathcal{C}_n^k(x)$ and $\mathcal{T}_n^k(x)$, both of which have definitions of the form
\[
	\mathcal{P}(x) = \mathcal{P}(x; w) = \sum_{\lambda \in \mathcal{\mathcal{P}}} x^{w(\lambda)}
\]
for a pair $(\mathcal{P}, w)$, where $\mathcal{P}$ is a set of combinatorial objects and $w: \mathcal{P} \to \Z_{\geq 0}$ is a suitable weight function. In this section, we introduce two additional combinatorial polynomials, $\widetilde{\mathcal{T}}_n^k(x)$ and $\mathscr{T}_n^k(x)$, and show that all polynomials coincide. In particular, we construct a sequence of bijections $\mathcal{T}_n^k \to \widetilde{\mathcal{T}}_n^k \to \mathscr{T}_n^k \to \mathcal{C}_n^k$ preserving the weight.

\subsection{Packed alternative tableaux}\label{s3-1}

Let $\mathcal{T}_n^k$ be the set of alternative tableaux of rectangular shape of size $n \times k$ as in \cref{AltTab-def}. We consider another weight $\wtl: \mathcal{T}_n^k \to \Z_{\geq 0}$ defined by the number of columns that contain $\leftarrow$ but does not contain $\downarrow$.

\begin{theorem}
	The polynomial $\mathcal{T}_n^k(x; \wtl) = \sum_{\lambda \in \mathcal{T}_n^k} x^{\wtl(\lambda)}$ coincides with $\widehat{\scB}_n^k(x)$.
\end{theorem}

\begin{proof}
	We can check that the polynomials $\mathcal{T}_n^k(x; \wtl)$ satisfy the recursion in \cref{recurrence}.
\end{proof}

The packed alternative tableaux introduced by Nadeau~\cite{Nadeau2011} complement alternative tableaux by adding lacking arrows.   

\begin{definition}
	A \emph{packed alternative tableau} of rectangular shape of size $n \times k$ is a rectangle of size $(n+1) \times (k+1)$ with a partial filling of the cells with left arrows and down arrows, such that 
	\begin{enumerate}
		\item all cells pointed by an arrow are empty, 
		\item each row (resp. column) except for the bottom row (resp. the leftmost column) contains exactly one left arrow $\leftarrow$ (resp. exactly one down arrow $\downarrow$), and
		\item the bottom row (resp. the leftmost column) does not contain $\leftarrow$ (resp. $\downarrow$).
	\end{enumerate}
	We let $\widetilde{\mathcal{T}}_n^k$ denote the set of all packed alternative tableaux of rectangular shape of size $n \times k$. For each $\lambda \in \widetilde{\mathcal{T}}_n^k$, the weight $\wttl(\lambda)$ counts the number of columns that contain $\leftarrow$ and $\downarrow$ in the bottom row.
\end{definition}

By cutting out the bottom row and the leftmost column of a packed alternative tableau of size $n \times k$, we obtain an alternative tableau of size $n \times k$. It is clear that the operation defines a bijection $(\widetilde{\mathcal{T}}_n^k, \wttl) \to (\mathcal{T}_n^k, \wtl)$. Thus, the polynomial
\begin{align}\label{packed-alt-polynomial}
	\widetilde{\mathcal{T}}_n^k(x) = \widetilde{\mathcal{T}}_n^k(x; \wttl) = \sum_{\lambda \in \widetilde{\mathcal{T}}_n^k} x^{\wttl(\lambda)}
\end{align}
coincides with $\mathcal{T}_n^k(x; \wtl)$, i.e., $\widetilde{\mathcal{T}}_n^k(x) = \widehat{\scB}_n^k(x)$.

\begin{example}
	The $\lambda \in \mathcal{T}_7^6$ given in \cref{AltTab76} corresponds to the following packed alternative tableau and has a weight of $\wttl(\lambda)=2$.
	\begin{figure}[H]
	\begin{tikzpicture}
		\draw (-1/2,-1/2)--(3,-1/2); \draw (-1/2,0)--(0,0); \draw [double](0,0)--(3,0); \draw (-1/2,1/2)--(3,1/2); \draw (-1/2,2/2)--(3,2/2); \draw (-1/2,3/2)--(3,3/2); \draw (-1/2,4/2)--(3,4/2); \draw (-1/2,5/2)--(3,5/2); \draw (-1/2,6/2)--(3,6/2); \draw (-1/2,7/2)--(3,7/2);
		\draw (-1/2,-1/2)--(-1/2,7/2); \draw (0,-1/2)--(0,0); \draw [double](0,0)--(0,7/2); \draw (1/2,-1/2)--(1/2,7/2); \draw (2/2,-1/2)--(2/2,7/2); \draw (3/2,-1/2)--(3/2,7/2); \draw (4/2,-1/2)--(4/2,7/2); \draw (5/2,-1/2)--(5/2,7/2); \draw (6/2,-1/2)--(6/2,7/2); 
		
		\draw (6/2-1/4,7/2-1/4) node{$\leftarrow$}; \draw (6/2-1/4,6/2-1/4) node{$\leftarrow$}; \draw (4/2-1/4,5/2-1/4) node{$\leftarrow$}; \draw (4/2-1/4,4/2-1/4) node{$\downarrow$}; \draw (3/2-1/4, 2/2-1/4) node{$\leftarrow$}; \draw (3/2-1/4, 1/2-1/4) node{$\downarrow$}; \draw (2/2-1/4, 4/2-1/4) node{$\downarrow$}; \draw (1/2-1/4, 3/2-1/4) node{$\leftarrow$}; \draw (1/2-1/4, 1/2-1/4) node{$\leftarrow$};
		
		\draw (-1/4, 3/2+1/4) node{$\leftarrow$}; \draw (1/4,-1/4) node{$\downarrow$}; \draw (2+1/4, -1/4) node{$\downarrow$}; \draw (5/2+1/4, -1/4) node{$\downarrow$};
	\end{tikzpicture}
	\caption{Packed alternative tableau of size $7 \times 6$.}
	\label{pAltTab76}
	\end{figure}
\end{example}

\subsection{Double alternative trees}\label{s3-2}

Alternative trees and forests were studied by Nadeau~\cite{Nadeau2011}. Based on this idea, we consider a pair of alternative trees and introduce a suitable weight to the trees.

\begin{definition}
	A \emph{double alternative tree} of size $n \times k$ is a pair of labeled rooted trees $(T_1, T_2)$, such that
	\begin{enumerate}
		\item the vertex set satisfies $V(T_1) \sqcup V(T_2) = \{\piros{0}, \piros{1}, \dots, \piros{n}, \kek{0}, \kek{1}, \dots, \kek{k}\}$,
		\item the roots of the trees $T_1$ and $T_2$ are $\piros{0}$ and $\kek{0}$, respectively,
		\item all children of each red (resp. blue) vertex are blue (resp. red), and
		\item for each vertex, its descendants have a different color or are larger than the vertex.
	\end{enumerate}
	We let $\mathscr{T}_n^k$ denote the set of all double alternative trees of size $n \times k$.
\end{definition}

\begin{example}
	This is an example of double alternative trees.
	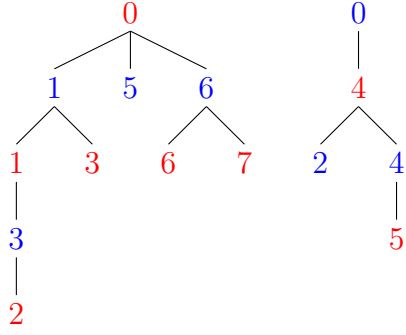
\begin{figure}[H]
	\begin{tikzpicture}
		\draw (3/2,4) node{$\piros{0}$}; \draw (9/2, 4) node{$\kek{0}$}; 
		
		\draw (3/2,4-1/4)--(1/2,3+1/4); \draw(3/2,4-1/4)--(3/2,3+1/4); \draw(3/2,4-1/4)--(5/2,3+1/4); \draw(9/2,4-1/4)--(9/2,3+1/4);
		
		\draw (1/2,3) node{$\kek{1}$}; \draw (3/2,3) node{$\kek{5}$}; \draw (5/2,3) node{$\kek{6}$}; \draw (9/2,3) node{$\piros{4}$}; 
		
		\draw (1/2,3-1/4)--(0,2+1/4); \draw (1/2,3-1/4)--(1,2+1/4); \draw (5/2,3-1/4)--(2,2+1/4); \draw (5/2,3-1/4)--(3,2+1/4); \draw (9/2,3-1/4)--(4,2+1/4); \draw (9/2,3-1/4)--(5,2+1/4); 
		
		\draw (0,2) node{$\piros{1}$}; \draw (1,2) node{$\piros{3}$}; \draw (2,2) node{$\piros{6}$}; \draw (3,2) node{$\piros{7}$}; \draw (4,2) node{$\kek{2}$}; \draw (5,2) node{$\kek{4}$}; 
		
		\draw (0,2-1/4)--(0,1+1/4); \draw (5,2-1/4)--(5,1+1/4);
		
		\draw (0,1) node{$\kek{3}$}; \draw (5,1) node{$\piros{5}$};
		
		\draw (0,1/4)--(0,1-1/4); \draw (0,0) node{$\piros{2}$};
	\end{tikzpicture}
	\caption{Double alternative tree of size $7 \times 6$.}
	\label{douAltTree76}
	\end{figure}
\end{example}

For each double alternative tree $\lambda \in \mathscr{T}_n^k$, we define its weight $\wtc(\lambda)$ by the number of non-leaf (blue) children of $\piros{0}$. Here, a vertex is called a \emph{leaf} if it does not have any child. For instance, the weight of the above $\lambda \in \mathscr{T}_7^6$ is $\wtc(\lambda) = \#\{\kek{1},\kek{6}\} = 2$. 

\begin{theorem}\label{double-alt-tree-poly}
	The polynomial $\mathscr{T}_n^k(x) = \mathscr{T}_n^k(x; \wtc)$ coincides with the polynomial $\widetilde{\mathcal{T}}_n^k(x)$ defined in $\cref{packed-alt-polynomial}$, i.e., $\mathscr{T}_n^k(x) = \widehat{\scB}_n^k(x)$.
\end{theorem}

\begin{proof}
	We can easily check that the map $\widetilde{\mathcal{T}}_n^k \to \mathscr{T}_n^k$ defined by
	
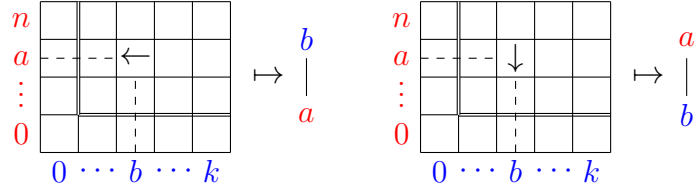
\begin{figure}[H]
	\begin{tikzpicture}
		\draw (-1/2,0)--(0,0); \draw [double](0,0)--(2,0); \draw (-1/2,-1/2)--(2,-1/2); \draw (-1/2,1/2)--(2,1/2); \draw (-1/2,1)--(2,1); \draw (-1/2,3/2)--(2,3/2); 
		\draw (0,-1/2)--(0,0); \draw [double](0,0)--(0,3/2); \draw (-1/2,-1/2)--(-1/2,3/2); \draw (1/2,-1/2)--(1/2,3/2); \draw (1,-1/2)--(1,3/2); \draw (3/2,-1/2)--(3/2,3/2); \draw (2,-1/2)--(2,3/2); 
		\draw (-3/4,-1/4)node{$\piros{0}$}; \draw (-3/4,1/4+1/8)node{$\piros{\vdots}$}; \draw (-3/4,3/4)node{$\piros{a}$}; \draw (-3/4,5/4)node{$\piros{n}$}; \draw (-1/4, -3/4)node{$\kek{0}$}; \draw (1/4, -3/4)node{$\kek{\cdots}$}; \draw (3/4, -3/4)node{$\kek{b}$}; \draw (5/4, -3/4)node{$\kek{\cdots}$}; \draw (7/4, -3/4)node{$\kek{k}$}; \draw (3/4,3/4)node{$\leftarrow$}; \draw [dashed](-1/2,3/4)--(1/2,3/4); \draw [dashed](3/4,-1/2)--(3/4,1/2); 
		
		\draw(5/2, 1/2)node{$\mapsto$}; 
		
		\draw (3,0)node{$\piros{a}$}; \draw (3,1)node{$\kek{b}$}; \draw (3,1/4)--(3,3/4);

		\draw (-1/2+5,0)--(0+5,0); \draw [double](0+5,0)--(2+5,0); \draw (-1/2+5,-1/2)--(2+5,-1/2); \draw (-1/2+5,1/2)--(2+5,1/2); \draw (-1/2+5,1)--(2+5,1); \draw (-1/2+5,3/2)--(2+5,3/2); 
		\draw (0+5,-1/2)--(0+5,0); \draw [double](0+5,0)--(0+5,3/2); \draw (-1/2+5,-1/2)--(-1/2+5,3/2); \draw (1/2+5,-1/2)--(1/2+5,3/2); \draw (1+5,-1/2)--(1+5,3/2); \draw (3/2+5,-1/2)--(3/2+5,3/2); \draw (2+5,-1/2)--(2+5,3/2); 
		\draw (-3/4+5,-1/4)node{$\piros{0}$}; \draw (-3/4+5,1/4+1/8)node{$\piros{\vdots}$}; \draw (-3/4+5,3/4)node{$\piros{a}$}; \draw (-3/4+5,5/4)node{$\piros{n}$}; \draw (-1/4+5, -3/4)node{$\kek{0}$}; \draw (1/4+5, -3/4)node{$\kek{\cdots}$}; \draw (3/4+5, -3/4)node{$\kek{b}$}; \draw (5/4+5, -3/4)node{$\kek{\cdots}$}; \draw (7/4+5, -3/4)node{$\kek{k}$}; \draw (3/4+5,3/4)node{$\downarrow$}; \draw [dashed](-1/2+5,3/4)--(1/2+5,3/4); \draw [dashed](3/4+5,-1/2)--(3/4+5,1/2); 
		
		\draw(5/2+5, 1/2)node{$\mapsto$}; 
		
		\draw (3+5,0)node{$\kek{b}$}; \draw (3+5,1)node{$\piros{a}$}; \draw (3+5,1/4)--(3+5,3/4);
	\end{tikzpicture}
	\end{figure}
	is a bijection $(\widetilde{\mathcal{T}}_n^k, \wttl) \to (\mathscr{T}_n^k, \wtc)$.
\end{proof}

Under the above bijection, the packed alternative tableau given in \cref{pAltTab76} corresponds to the double alternative tree in \cref{douAltTree76}.

\begin{theorem}
	There is a bijection $(\mathscr{T}_n^k, \wtc) \to (\mathcal{C}_n^k, \wclr)$, i.e., the polynomial $\mathscr{T}_n^k(x)$ coincides with the polynomial $\mathcal{C}_n^k(x)$ defined in $\cref{def-Callan-polynomial}$.
\end{theorem}

\begin{proof}
	We define a bijection $\phi$ from the set of labeled rooted trees to the set of strings inductively. For a singleton $T = v$, we put $\phi(T) = v$. Let $T$ be the following rooted tree:
	
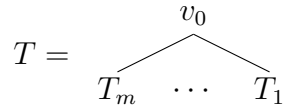
\begin{figure}[H]
	\begin{tikzpicture}
		\draw (2,1)node{$v_0$}; \draw (1,1/4)--(2,1-1/4); \draw (3,1/4)--(2,1-1/4);
		\draw (0,1/2)node{$T = $}; \draw (1,0)node{$T_m$}; \draw (2,0)node{$\cdots$}; \draw (3,0)node{$T_1$};
	\end{tikzpicture}
	\end{figure}
	In this expression, assume that $T_1, \dots, T_m$ are rooted trees, the roots $R(T_1), \dots, R(T_m)$ of which satisfy the condition $R(T_m) < \cdots < R(T_1)$. We define $\phi(T) = \phi(v_0) \phi(T_1) \cdots \phi(T_m)$. The desired bijection is given by $\piros{0} S_1 = \phi(T_1)$ and $\kek{0} S_2 = \phi(T_2)$, which preserves the weight.
\end{proof}

\begin{example}\label{doubleCallan-barredCallan}
Under the above bijection, the example given in \cref{douAltTree76} corresponds to $(S_1, S_2) = (\kek{6} \piros{76} \kek{51} \piros{31} \kek{3} \piros{2}, \piros{4} \kek{4} \piros{5} \kek{2})$.
\end{example}

In conclusion, we obtain another bijection between $\mathcal{C}_n^k$ and $\mathcal{T}_n^k$ by a sequence of bijections
\begin{align}\label{sequence-bijections}
	(\mathcal{T}_n^k, \wtl) \to (\widetilde{\mathcal{T}}_n^k, \wttl) \to (\mathscr{T}_n^k, \wtc) \to (\mathcal{C}_n^k, \wclr).
\end{align}

By the bijection in \cref{direct-bijection-CT}, the alternative tableau in \cref{AltTab76} corresponds to the double Callan permutation $(\kek{6}\piros{21} \kek{4}\piros{4} \kek{5}, \piros{76}\kek{32} \piros{53}\kek{1})$ as explained in \cref{ExaExa}. On the other hand, by the bijection in \cref{sequence-bijections}, the alternative tableau corresponds to $(\kek{6} \piros{76} \kek{51} \piros{31} \kek{3} \piros{2}, \piros{4} \kek{4} \piros{5} \kek{2})$ as in \cref{doubleCallan-barredCallan}. The difference arises from the existence of two weight functions $\wtst$ and $\wtl$ for the set of alternative tableaux $\mathcal{T}_n^k$.

By translating the weight $\wtst$ via the bijections $\mathcal{T}_n^k \to \widetilde{\mathcal{T}}_n^k \to \mathscr{T}_n^k \to \mathcal{C}_n^k$, we can obtain new weight functions for $\widetilde{\mathcal{T}}_n^k, \mathscr{T}_n^k$, and $\mathcal{C}_n^k$. For example, the new weight $\wcrb: \mathcal{C}_n^k \to \Z_{\geq 0}$ is defined as follows:

\begin{definition}
	For each double Callan permutation $\lambda = (S_1, S_2) \in \mathcal{C}_n^k$,
	\begin{enumerate}
		\item if the string $S_2$ does not start from $\piros{n}$, then mark the blue element just before $\piros{n}$, and if $S_2$ starts from $\piros{n}$, then we stop the steps,
		\item consider the next largest red element, 
			\begin{itemize}
				\item[(i)] if the element is the leading element of $S_2$, then we stop the steps,
				\item[(ii)] if the element is after a blue element and the blue element is smaller than the last marked element, then we mark the blue element,
				\item[(iii)] otherwise, we do nothing, and
			\end{itemize}
		\item repeat Step (2) until we reach $\piros{1}$ or until the step stops.
	\end{enumerate}
	Then, we define $\wcrb(\lambda)$ by the number of marked blue elements.
\end{definition}

\begin{example}
	Let $\lambda_0 \in \mathcal{T}_7^6$ be as in \cref{AltTab76}. By the bijections of \cref{sequence-bijections}, $\lambda_0$ corresponds to the double Callan permutation $\lambda = (\kek{6} \piros{76} \kek{51} \piros{31} \kek{3} \piros{2}, \piros{4} \kek{4} \piros{5} \kek{2})$, as in \cref{doubleCallan-barredCallan}.
	
	We first mark $\kek{6}$. Since the next largest red element $\piros{6}$ is located after $\piros{7}$, we ignore this element. The next $\piros{5}$ is after a blue element. Since $\kek{4}$ is smaller than the last marked $\kek{6}$, we mark $\kek{4}$. The next $\piros{4}$ is the leading element of $S_2$. Thus, we stop the steps here. The weight is given by $\wcrb(\lambda) = \#\{\kek{6}, \kek{4}\} = 2$, which coincides with $\wtst(\lambda_0)$. In particular, the indicated left arrows in \cref{AltTab76} are in $4$th and $6$th columns.
\end{example}

\begin{corollary}
	We have a bijection $(\mathcal{T}_n^k, \wtst) \to (\mathcal{C}_n^k, \wcrb)$, i.e., $\mathcal{C}_n^k(x; \wcrb) = \mathcal{T}_n^k (x; \wtst) = \widehat{\mathscr{B}}_n^k(x)$.
\end{corollary}

\section{Further combinatorial models and weights}\label{s4}

In this section, we provide another combinatorial model $(\mathcal{E}_n^k, \welr)$ and prove that the polynomial $\mathcal{E}_n^k(x; \welr)$ is equal to $\widehat{\mathscr{B}}_n^k(x)$. We here explain two types of proofs. One type is by checking the recursion \cref{recurrence}, and the other type is by constructing a bijection $(\mathcal{C}_n^k, \wcRL) \to (\mathcal{E}_n^k, \welr)$ with a new weight $\wcRL: \mathcal{C}_n^k \to \Z_{\geq 0}$.

\subsection{Excedance set of permutations}\label{s4-1}

We introduce the fifth combinatorial set for the symmetrized poly-Bernoulli polynomial $\widehat{\scB}_n^k(x)$ using an excedance set of a permutation, which was studied by Ehrenborg--Steingr\'{i}msson~\cite{EhrenborgSteingrimsson2000}. 

\begin{definition}
	Let $[n] = \{1, 2, \dots, n\}$. An \emph{excedance set} of a permutation $\lambda: [n] \twoheadrightarrow [n]$ is defined by $E (\lambda) = \{i \in [n] \mid \lambda(i) > i\}$. For non-negative integers $n, k \geq 0$, we set $\mathcal{E}_n^k = \{\lambda: [n+k+1] \twoheadrightarrow [n+k+1] \mid E(\lambda) = [n]\}$.
\end{definition}

\begin{example}
	The following lists all elements $\lambda = \smat{1 & 2 & 3 & 4 \\ \lambda(1) & \lambda(2) & \lambda(3) & \lambda(4)}$ in $\mathcal{E}_2^1$.
	\begin{align*}
		&\pmat{1 & 2 & 3 & 4 \\ 2 & 3 & 1 &4}, \pmat{1 & 2 & 3 & 4 \\ 2 & 4 & 1 & 3}, \pmat{1 & 2 & 3 & 4 \\ 2 & 4 & 3 & 1}, \pmat{1 & 2 & 3 & 4 \\ 3 & 4 & 1 & 2},\\
		&\pmat{1 & 2 & 3 & 4 \\ 3 & 4 & 2 & 1}, \pmat{1 & 2 & 3 & 4 \\ 4 & 3 & 1 & 2}, \pmat{1 & 2 & 3 & 4 \\ 4 & 3 & 2 & 1}
	\end{align*}
\end{example}

We define the weight function $\welr: \mathcal{E}_n^k \to \Z_{\geq 0}$ by the left-to-right minimum.

\begin{definition}
	For each $\lambda \in \mathcal{E}_n^k$, we consider $n < i \leq n+k+1$ such that if $n < j < i$, then $\lambda(i) < \lambda(j)$. We let $\welr(\lambda)$ denote the number of such $i$ reduced by one.
\end{definition}

For instance, the permutation 
\begin{align}\label{Excedance76}
	\lambda = \pmat{1 & 2 & 3 & 4 & 5 & 6 & 7 \\ 7 & 5 & 3 & 2 & 4 & 6 & 1} \in \mathcal{E}_2^4
\end{align}
has the weight $\welr(\lambda) = \#\{3, 4, 7\} - 1 = 2$.

We can express elements in $\mathcal{E}_n^k$ using the following chessboard of size $n+k+1$ with cracked squares, (see Clark--Ehrenborg~\cite{ClarkEhrenborg2010}).

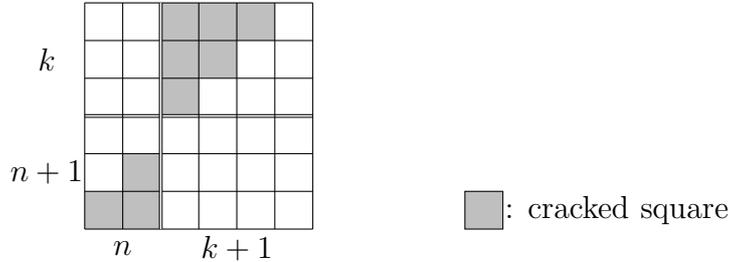
\begin{figure}[H]
\begin{tikzpicture}
	\fill [lightgray] (0,0)--(1,0)--(1,1)--(1/2,1)--(1/2,1/2)--(0,1/2)--(0,0)--cycle;
	\fill [lightgray] (1,3)--(5/2,3)--(5/2,5/2)--(4/2,5/2)--(4/2,4/2)--(3/2,4/2)--(3/2,3/2)--(2/2,3/2)--(1,3)--cycle;

	\draw (0,0)--(3,0); \draw(0,1/2)--(3,1/2); \draw(0,2/2)--(3,2/2); \draw[double](0,3/2)--(3,3/2); \draw(0,4/2)--(3,4/2); \draw(0,5/2)--(3,5/2); \draw(0,3)--(3,3);
	\draw (0,0)--(0,3); \draw (1/2,0)--(1/2,3); \draw[double] (2/2,0)--(2/2,3); \draw (3/2,0)--(3/2,3); \draw (4/2,0)--(4/2,3); \draw (5/2,0)--(5/2,3); \draw (6/2,0)--(6/2,3); 
	\draw (1/2,-1/4)node{$n$}; \draw (4/2,-1/4)node{$k+1$}; \draw (-1/2,3/4)node{$n+1$}; \draw (-1/2,9/4)node{$k$}; 
	
	\fill [lightgray] (5,0)--(5+1/2,0)--(5+1/2,1/2)--(5,1/2)--cycle; \draw (5,0)--(5+1/2,0)--(5+1/2,1/2)--(5,1/2)--cycle; \draw (7,1/4)node{: cracked square};
\end{tikzpicture}
\end{figure}

Then, the ways of placing $n+k+1$ non-attacking rooks on a cracked chessboard correspond to the elements of $\mathcal{E}_n^k$. For instance, the element $\lambda \in \mathcal{E}_7^6$ given in \cref{Excedance76} is expressed as follows:

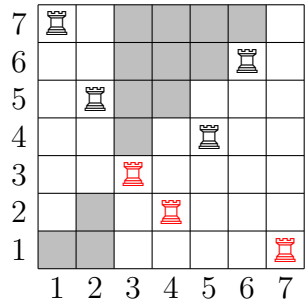
\begin{figure}[H]
\begin{tikzpicture}
	\fill [lightgray] (0,0)--(1,0)--(1,1)--(1/2,1)--(1/2,1/2)--(0,1/2)--cycle;
	\fill [lightgray] (7/2,7/2)--(1,7/2)--(1,3/2)--(3/2,3/2)--(3/2,2)--(2,2)--(2,5/2)--(5/2,5/2)--(5/2,3)--(3,3)--(3,7/2)--cycle;
	
	\draw (0,0)--(7/2,0); \draw (0,1/2)--(7/2,1/2); \draw (0,2/2)--(7/2,2/2); \draw (0,3/2)--(7/2,3/2); \draw (0,4/2)--(7/2,4/2); \draw (0,5/2)--(7/2,5/2); \draw (0,6/2)--(7/2,6/2); \draw (0,7/2)--(7/2,7/2); 
	\draw (0,0)--(0,7/2); \draw (1/2,0)--(1/2,7/2); \draw (2/2,0)--(2/2,7/2); \draw (3/2,0)--(3/2,7/2); \draw (4/2,0)--(4/2,7/2); \draw (5/2,0)--(5/2,7/2); \draw (6/2,0)--(6/2,7/2); \draw (7/2,0)--(7/2,7/2); 
	
	\draw (0+1/4, 3+1/4)node{$\rook$}; \draw (1/2+1/4, 2+1/4)node{$\rook$}; \draw (2/2+1/4, 1+1/4)node{$\piros{\rook}$}; \draw (3/2+1/4, 1/2+1/4)node{$\piros{\rook}$}; \draw (4/2+1/4, 3/2+1/4)node{$\rook$}; \draw (5/2+1/4, 5/2+1/4)node{$\rook$}; \draw (6/2+1/4, 0+1/4)node{$\piros{\rook}$}; 
	
	\draw (0+1/4,-1/4)node{$1$}; \draw (1/2+1/4,-1/4)node{$2$}; \draw (2/2+1/4,-1/4)node{$3$}; \draw (3/2+1/4,-1/4)node{$4$}; \draw (4/2+1/4,-1/4)node{$5$}; \draw (5/2+1/4,-1/4)node{$6$}; \draw (6/2+1/4,-1/4)node{$7$}; 
	
	\draw (-1/4,0+1/4)node{$1$}; \draw (-1/4,1/2+1/4)node{$2$}; \draw (-1/4,2/2+1/4)node{$3$}; \draw (-1/4,3/2+1/4)node{$4$}; \draw (-1/4,4/2+1/4)node{$5$}; \draw (-1/4,5/2+1/4)node{$6$}; \draw (-1/4,6/2+1/4)node{$7$}; 
\end{tikzpicture}
\caption{The expression of $\lambda =\pmat{1 & 2 & 3 & 4 & 5 & 6 & 7 \\ 7 & 5 & 3 & 2 & 4 & 6 & 1} \in \mathcal{E}_2^4$ with the indication of its weight.}
\label{Example-E76}
\end{figure}

\begin{theorem}\label{Excedance-theorem}
	Let $\mathcal{E}_n^0 (x) = \mathcal{E}_0^k(x) = 1$. For any integers $n, k \geq 0$, the polynomial
	\[
		\mathcal{E}_n^k(x) = \mathcal{E}_n^k (x; \welr) = \sum_{\lambda \in \mathcal{E}_n^k} x^{\welr(\lambda)}
	\]
	coincides with the polynomial $\widehat{\scB}_n^k(x)$.
\end{theorem}

To prove this theorem, we define the auxiliary map $\psi: \mathcal{E}_n^k \to \mathcal{E}_{\leq n}^{k-1} := \bigcup_{i=0}^n \mathcal{E}_i^{k-1}$ as follows. 

\begin{definition}\label{psi-definition}
For $\lambda \in \mathcal{E}_n^k$, 
\begin{itemize}
	\item[(1)] If $\lambda(n+k+1) = n+k+1$, then $\psi(\lambda) \in \mathcal{E}_n^{k-1}$ with $\psi(\lambda) (i) = \lambda(i)$ for any $1 \leq i \leq n+k$.
	
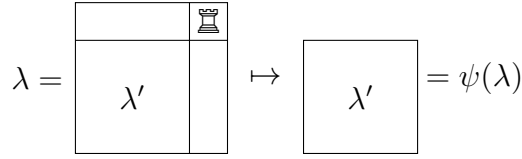
	\begin{figure}[H]
\begin{tikzpicture}
	\draw (-1/2,1) node {$\lambda = $}; \draw (0,0)--(2,0)--(2,2)--(0,2)--cycle; \draw (3/2,0)--(3/2,2); \draw (0,3/2)--(2,3/2); \draw (3/4, 3/4) node {$\lambda'$}; \draw (3/2+1/4, 3/2+1/4) node {$\rook$}; 
	\draw (5/2,1) node{$\mapsto$}; 
	\draw (3,0)--(9/2,0)--(9/2,3/2)--(3,3/2)--cycle; \draw (3+3/4,3/4) node{$\lambda'$}; 
	\draw (9/2+3/4,1) node{$= \psi(\lambda)$};
\end{tikzpicture}
\end{figure}
	
	\item[(2)] If $\lambda(a) = n+k+1, \lambda(n+k+1) = b$, and $b > a$, then $\psi(\lambda) \in \mathcal{E}_n^{k-1}$ with $\psi(\lambda) (a) = b$ and $\psi(\lambda)(i) = \lambda(i)$ otherwise.
	
	
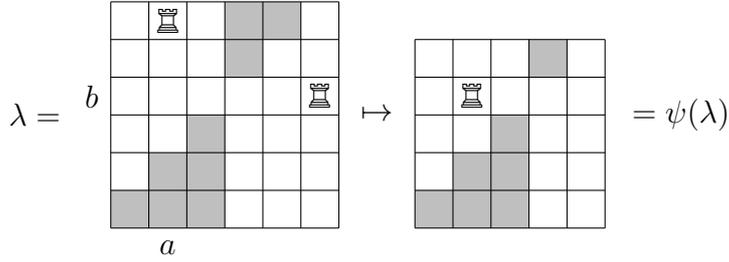
\begin{figure}[H]
	\begin{tikzpicture}
		\draw (-1,3/2) node {$\lambda = $}; \fill [lightgray] (0,0)--(3/2,0)--(3/2,3/2)--(1,3/2)--(1,1)--(1/2,1)--(1/2,1/2)--(0,1/2)--cycle; \fill [lightgray] (5/2,3)--(3/2,3)--(3/2,2)--(2,2)--(2,5/2)--(5/2,5/2)--cycle; \draw (0,0)--(3,0); \draw (0,1/2)--(3,1/2); \draw (0,2/2)--(3,2/2); \draw (0,3/2)--(3,3/2); \draw (0,4/2)--(3,4/2); \draw (0,5/2)--(3,5/2); \draw (0,6/2)--(3,6/2); 
		\draw (0,0)--(0,3); \draw (1/2,0)--(1/2,3); \draw (2/2,0)--(2/2,3); \draw (3/2,0)--(3/2,3); \draw (4/2,0)--(4/2,3); \draw (5/2,0)--(5/2,3); \draw (6/2,0)--(6/2,3); 
		\draw (1/2+1/4,-1/4) node{$a$}; \draw (-1/4, 3/2+1/4) node{$b$}; \draw (1/2+1/4, 5/2+1/4) node {$\rook$}; \draw (5/2+1/4, 3/2+1/4) node {$\rook$};
	
		\draw (7/2, 3/2) node {$\mapsto$}; 
	
		\fill [lightgray] (4,0)--(4+3/2,0)--(4+3/2,3/2)--(4+1,3/2)--(4+1,1)--(4+1/2,1)--(4+1/2,1/2)--(4+0,1/2)--cycle; \fill [lightgray] (11/2,2)--(12/2,2)--(12/2,5/2)--(11/2,5/2)--cycle; \draw (4,0)--(13/2,0); \draw (4,1/2)--(13/2,1/2); \draw (4,2/2)--(13/2,2/2); \draw (4,3/2)--(13/2,3/2); \draw (4,4/2)--(13/2,4/2); \draw (4,5/2)--(13/2,5/2); 
		\draw (4,0)--(4,5/2); \draw (9/2,0)--(9/2,5/2); \draw (10/2,0)--(10/2,5/2); \draw (11/2,0)--(11/2,5/2); \draw (12/2,0)--(12/2,5/2); \draw (13/2,0)--(13/2,5/2); 
		\draw (9/2+1/4, 3/2+1/4) node {$\rook$};
		
		\draw (15/2, 3/2) node {$= \psi(\lambda)$}; 
	\end{tikzpicture}
	\end{figure}
	
	\item[(3)] If $\lambda(a) = n+k+1, \lambda(n+k+1) = b$, and $b \leq a \leq n$, then we set $D = \{b \leq i < a \mid \lambda(i) = i+1\} \subset [n-1]$. Then, we define $\psi'(\lambda)$ by $\psi'(\lambda) = \lambda$ on $[n+k] - (D \cup \{a\})$. From $\psi'(\lambda)$, we construct $\psi(\lambda) \in \mathcal{E}_{n-|D|-1}^{k-1}$ by rearranging numbers from $1$.
	
	
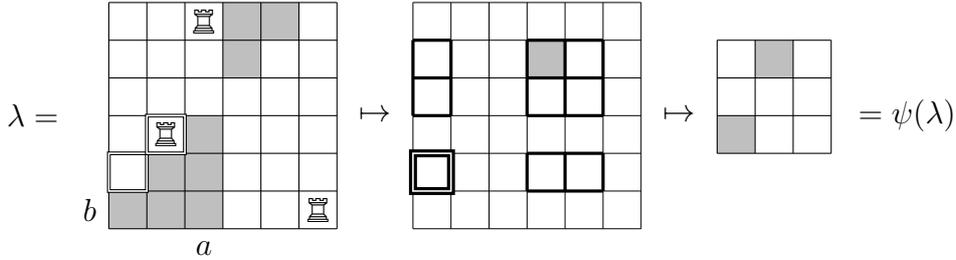
\begin{figure}[H]
	\begin{tikzpicture}
		\draw (-1,3/2) node {$\lambda = $}; \fill [lightgray] (0,0)--(3/2,0)--(3/2,3/2)--(1,3/2)--(1,1)--(1/2,1)--(1/2,1/2)--(0,1/2)--cycle; \fill [lightgray] (5/2,3)--(3/2,3)--(3/2,2)--(2,2)--(2,5/2)--(5/2,5/2)--cycle; \draw (0,0)--(3,0); \draw (0,1/2)--(3,1/2); \draw (0,2/2)--(3,2/2); \draw (0,3/2)--(3,3/2); \draw (0,4/2)--(3,4/2); \draw (0,5/2)--(3,5/2); \draw (0,6/2)--(3,6/2); 
		\draw (0,0)--(0,3); \draw (1/2,0)--(1/2,3); \draw (2/2,0)--(2/2,3); \draw (3/2,0)--(3/2,3); \draw (4/2,0)--(4/2,3); \draw (5/2,0)--(5/2,3); \draw (6/2,0)--(6/2,3); 
		\draw (1+1/4,-1/4) node{$a$}; \draw (-1/4, 1/4) node{$b$}; \draw (1+1/4, 5/2+1/4) node {$\rook$}; \draw (5/2+1/4, 1/4) node {$\rook$}; \draw (1/2+1/4,1+1/4) node{$\rook$}; \draw[double] (0,1/2)--(1/2,1/2)--(1/2,1)--(0,1)--cycle; \draw[double] (1/2,1)--(1,1)--(1,3/2)--(1/2,3/2)--cycle;
	
		\draw (7/2, 3/2) node {$\mapsto$}; 
		
		\draw[very thin] (4,0)--(7,0); \draw[very thin] (4,1/2)--(7,1/2); \draw[very thin] (4,2/2)--(7,2/2); \draw[very thin] (4,3/2)--(7,3/2); \draw[very thin] (4,4/2)--(7,4/2); \draw[very thin] (4,5/2)--(7,5/2); \draw[very thin] (4,6/2)--(7,6/2); 
		\draw[very thin] (4,0)--(4,3); \draw[very thin] (9/2,0)--(9/2,3); \draw[very thin] (10/2,0)--(10/2,3); \draw[very thin] (11/2,0)--(11/2,3); \draw[very thin] (12/2,0)--(12/2,3); \draw[very thin] (13/2,0)--(13/2,3); \draw[very thin] (14/2,0)--(14/2,3); 
		
		\draw[double, very thick] (4,1/2)--(9/2,1/2)--(9/2,1)--(4,1)--cycle; \fill [lightgray] (11/2,2)--(12/2,2)--(12/2,5/2)--(11/2,5/2)--cycle; 
		\draw[very thick] (11/2,1/2)--(13/2,1/2); \draw[very thick] (11/2,2/2)--(13/2,2/2); \draw[very thick] (4,3/2)--(9/2,3/2); \draw[very thick](11/2,3/2)--(13/2,3/2); \draw[very thick] (4,4/2)--(9/2,4/2); \draw[very thick](11/2,4/2)--(13/2,4/2); \draw[very thick] (4,5/2)--(9/2,5/2); \draw[very thick](11/2,5/2)--(13/2,5/2);
		\draw[very thick] (4,3/2)--(4,5/2); \draw[very thick] (9/2,3/2)--(9/2,5/2); \draw[very thick] (11/2,1/2)--(11/2,1); \draw[very thick] (11/2,3/2)--(11/2,5/2); \draw[very thick] (12/2,1/2)--(12/2,1); \draw[very thick] (12/2,3/2)--(12/2,5/2); \draw[very thick] (13/2,1/2)--(13/2,1); \draw[very thick] (13/2,3/2)--(13/2,5/2); 
		
		\draw (15/2, 3/2) node {$\mapsto$}; 
		
		\fill[lightgray] (8,1)--(17/2,1)--(17/2,3/2)--(8,3/2)--cycle; \fill[lightgray] (17/2,2)--(9,2)--(9,5/2)--(17/2,5/2)--cycle; 
		\draw (8,1)--(19/2,1); \draw (8,3/2)--(19/2,3/2); \draw (8,2)--(19/2,2); \draw (8,5/2)--(19/2,5/2); \draw (8,1)--(8,5/2); \draw (17/2,1)--(17/2,5/2); \draw (9,1)--(9,5/2); \draw (19/2,1)--(19/2,5/2); 
		\draw (21/2, 3/2) node {$= \psi(\lambda)$}; 
	\end{tikzpicture}
	\end{figure}	
\end{itemize}
\end{definition}

\begin{proof}[Proof of $\cref{Excedance-theorem}$]
	We split the set $\mathcal{E}_n^k$ into disjoint subsets by the above conditions (1), (2), and (3). 
	
	(1) The function $\psi$ gives a bijection from $\{\lambda \in \mathcal{E}_n^k \mid \lambda(n+k+1) = n+k+1\}$ to $\mathcal{E}_n^{k-1}$, which preserves the weight. Thus, we have
	\[
		\sum_{\lambda \in \mathcal{E}_n^k|_{(1)}} x^{\welr(\lambda)} = \mathcal{E}_n^{k-1}(x).
	\]
	
	(2) In this case, the function $\psi: \{\lambda \in \mathcal{E}_n^k \mid (2)\} \to \mathcal{E}_n^{k-1}$ is $n$-to-$1$, which preserves the weight, i.e.,
	\[
		\sum_{\lambda \in \mathcal{E}_n^k|_{(2)}} x^{\welr(\lambda)} = n \mathcal{E}_n^{k-1}(x).
	\]
	
	(3) If $b = 1$, the function $\psi$ reduces the weight by one. In this case, we obtain
	\begin{align*}
		\sum_{\substack{\lambda \in \mathcal{E}_n^k|_{(3)} \\ b = 1}} x^{\welr(\lambda)} = x \sum_{a=1}^n \sum_{|D|=0}^{a-1} {a-1 \choose |D|} \mathcal{E}_{n-|D|-1}^{k-1} (x) = x \sum_{d=0}^{n-1} {n \choose d+1} \mathcal{E}_{n-d-1}^{k-1}(x).
	\end{align*}
	If $b > 1$, $\psi$ does not affect the weight. In a similar manner, we have
	\[
		\sum_{\substack{\lambda \in \mathcal{E}_n^k|_{(3)} \\ b > 1}} x^{\welr(\lambda)} = \sum_{2 \leq b \leq a \leq n} \sum_{|D|=0}^{a-b} {a-b \choose |D|} \mathcal{E}_{n-|D|-1}^{k-1} (x) = \sum_{d=0}^{n-2} {n \choose d+2} \mathcal{E}_{n-d-1}^{k-1}(x).
	\]
	By summation, we have the recurrence formula
	\[
		\mathcal{E}_n^k(x) = (n+1) \mathcal{E}_n^{k-1}(x) + x \sum_{d=0}^{n-1} {n \choose d} \mathcal{E}_{d}^{k-1}(x) + \sum_{d=1}^{n-1} {n \choose d-1} \mathcal{E}_d^{k-1}(x),
	\]
	which coincides with that in \cref{recurrence}.
\end{proof}


\subsection{Another proof by a combinatorial bijection}\label{s4-2}

First, we label the cracked chessboard as follows:
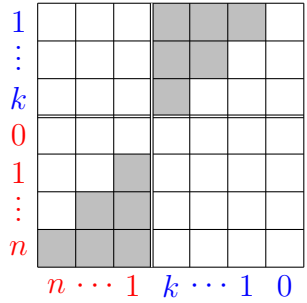
\begin{figure}[H]
\begin{tikzpicture}
	\fill [lightgray] (0,0)--(3/2,0)--(3/2,3/2)--(1,3/2)--(1,1)--(1/2,1)--(1/2,1/2)--(0,1/2)--cycle;
	\fill [lightgray] (3/2,7/2)--(6/2,7/2)--(6/2,6/2)--(5/2,6/2)--(5/2,5/2)--(4/2,5/2)--(4/2,4/2)--(3/2,4/2)--cycle;

	\draw (0,0)--(7/2,0); \draw(0,1/2)--(7/2,1/2); \draw(0,2/2)--(7/2,2/2); \draw(0,3/2)--(7/2,3/2); \draw[double](0,4/2)--(7/2,4/2); \draw(0,5/2)--(7/2,5/2); \draw(0,3)--(7/2,3); \draw(0,7/2)--(7/2,7/2);
	\draw (0,0)--(0,7/2); \draw (1/2,0)--(1/2,7/2); \draw (2/2,0)--(2/2,7/2); \draw[double] (3/2,0)--(3/2,7/2); \draw (4/2,0)--(4/2,7/2); \draw (5/2,0)--(5/2,7/2); \draw (6/2,0)--(6/2,7/2); \draw (7/2,0)--(7/2,7/2);
	\draw (1/4,-1/4) node {$\piros{n}$}; \draw (1/4+1/2,-1/4) node {$\piros{\cdots}$}; \draw (1/4+2/2,-1/4) node {$\piros{1}$}; \draw (1/4+3/2,-1/4) node {$\kek{k}$}; \draw (1/4+4/2,-1/4) node {$\kek{\cdots}$}; \draw (1/4+5/2,-1/4) node {$\kek{1}$}; \draw (1/4+6/2,-1/4) node {$\kek{0}$}; 
	\draw (-1/4, 1/4+0/2) node {$\piros{n}$}; \draw (-1/4, 1/4+1/2+1/8) node {$\piros{\vdots}$}; \draw (-1/4, 1/4+2/2) node {$\piros{1}$}; \draw (-1/4, 1/4+3/2) node {$\piros{0}$}; \draw (-1/4, 1/4+4/2) node {$\kek{k}$}; \draw (-1/4, 1/4+5/2+1/8) node {$\kek{\vdots}$}; \draw (-1/4, 1/4+6/2) node {$\kek{1}$}; 
\end{tikzpicture}
\end{figure}

For each double Callan permutation $(S_1,S_2) \in \mathcal{C}_n^k$, the placement of $n+k+1$ rooks on the chessboard are as follows:
We set $S=S_1\kek{0} S_2 \piros{0}$.
\begin{itemize}
	\item[(1)]For each adjacent same colored pair $\piros{xy}$ (resp. $\kek{xy}$) in $S$, place a rook on $\piros{y}$-row, $\piros{x}$-column (resp. $\kek{x}$-row, $\kek{y}$-column).
	\item[(2)]Let $S=x_1x_2\ldots x_{n+k+2}$. We perform the following operations in the order $i=1,2, \dots, n+k+1$:
	\begin{itemize}
		\item[(i)] If there is already a rook in the $x_i$-column, then we do nothing.
		\item[(ii)] If there is no rook in the $x_i$-column, then we place a rook at the $x_i$-column and the topmost row among the rows of a different color from $x_i$ without rooks.  
	\end{itemize}
\end{itemize}

\begin{example}
From the double Callan permutation $(S_1,S_2)=(\kek{21}\piros{3},\piros{2}\kek{3}\piros{1})$, we obtain the string $S=\kek{21}\piros{3}\kek{0}\piros{2}\kek{3}\piros{10}$. The corresponding placement of rooks is as follows:
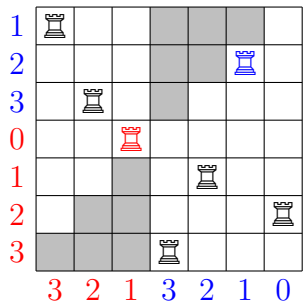
\begin{figure}[H]
	\begin{tikzpicture}
		\fill [lightgray] (0,0)--(3/2,0)--(3/2,3/2)--(1,3/2)--(1,1)--(1/2,1)--(1/2,1/2)--(0,1/2)--cycle;
		\fill [lightgray] (3/2,2)--(2,2)--(2,5/2)--(5/2,5/2)--(5/2,6/2)--(6/2,6/2)--(6/2,7/2)--(3/2,7/2)--cycle;
		\draw (0,0)--(7/2,0);
		\draw (0,1/2)--(7/2,1/2);
		\draw (0,2/2)--(7/2,2/2);
		\draw (0,3/2)--(7/2,3/2);
		\draw (0,4/2)--(7/2,4/2);
		\draw (0,5/2)--(7/2,5/2);
		\draw (0,6/2)--(7/2,6/2); 
		\draw (0,7/2)--(7/2,7/2); 
		\draw (0,0)--(0,7/2);
		\draw (1/2,0)--(1/2,7/2);
		\draw (2/2,0)--(2/2,7/2);
		\draw (3/2,0)--(3/2,7/2);
		\draw (4/2,0)--(4/2,7/2);
		\draw (5/2,0)--(5/2,7/2);
		\draw (6/2,0)--(6/2,7/2); 
		\draw (7/2,0)--(7/2,7/2); 
		\draw (1/4,-1/4) node{$\piros{3}$};
		\draw (1/2+1/4, -1/4) node{$\piros{2}$};
		\draw (1+1/4, -1/4) node{$\piros{1}$};
		\draw (3/2+1/4, -1/4) node {$\kek{3}$};
		\draw (2+1/4, -1/4) node {$\kek{2}$};
		\draw (5/2+1/4, -1/4) node {$\kek{1}$};
		\draw (3+1/4, -1/4) node {$\kek{0}$};
		\draw (-1/4,1/4) node{$\piros{3}$};
		\draw (-1/4, 1/2+1/4) node{$\piros{2}$};
		\draw (-1/4, 2/2+1/4) node{$\piros{1}$};
		\draw (-1/4, 3/2+1/4) node {$\piros{0}$};
		\draw (-1/4, 4/2+1/4) node {$\kek{3}$};
		\draw (-1/4, 5/2+1/4) node {$\kek{2}$};
		\draw (-1/4, 6/2+1/4) node {$\kek{1}$};
		\draw (0/2+1/4,6/2+1/4) node {$\rook$};
		\draw (1/2+1/4,4/2+1/4) node {$\rook$};
		\draw (2/2+1/4,3/2+1/4) node {$\piros{\rook}$};
		\draw (3/2+1/4,0/2+1/4) node {$\rook$};
		\draw (4/2+1/4,2/2+1/4) node {$\rook$};
		\draw (5/2+1/4,5/2+1/4) node {$\kek{\rook}$};
		\draw (6/2+1/4,1/2+1/4) node {$\rook$};
	\end{tikzpicture}
\end{figure}
\end{example}

\begin{theorem}
The correspondence is well-defined and is a bijection $\mathcal{C}_n^k \to \mathcal{E}_n^k$.
\end{theorem}

\begin{proof}
	If there are $r$ adjacent red pairs and $b$ adjacent blue pairs in the string $S$, then there are $n+1-r$ red substrings and $k+1-b$ blue substrings. By the definition of double Callan permutations, the equation $n+1-r = k+1-b$ holds.
	
	By the definition of Step (1), rooks are located on non-cracking squares. The equation $n-r = k-b$ guarantees the well-definedness of the map. Since we can define the inverse map, this map is bijective.
\end{proof}

Throughout the above bijection, the weight $\welr$ allows us to define a new weight $\wcRL$ for $\mathcal{C}_n^k$. 

\begin{definition}
For each $\lambda = (S_1, S_2) \in \mathcal{C}_n^k$, we consider the string $S = S_1 \kek{0} S_2 \piros{0}$ as before. Let $\ell$ be the number of blue substrings in $S$. We define the weight $\wcRL(\lambda)$ using the right-to-left maximum as follows:
\begin{itemize}
	\item[(1)]Consider the maximum of each blue substring in $S$ to obtain a sequence $\pi = \pi_1\cdots \pi_\ell$.
	\item[(2)]Count the number of $1 \leq i \leq \ell$ such that, if $i < j$, then $\pi_j<\pi_i$.
	\item[(3)]Subtract 1 from the number.
\end{itemize}	
\end{definition}

\begin{example}
For $\lambda = (\kek{6}\piros{21} \kek{4}\piros{4} \kek{5}, \piros{76}\kek{32} \piros{53}\kek{1})$, we have $S = \kek{6}\piros{21} \kek{4}\piros{4} \kek{50} \piros{76}\kek{32} \piros{53}\kek{1} \piros{0}$ and $\pi = \kek{64531}$. The weight of $\lambda$ is given by $\wcRL(\lambda) = \#\{\kek{6}, \kek{5}, \kek{3}, \kek{1}\} - 1 = 3$. Similarly, we have $\wcRL((\kek{21}\piros{3},\piros{2}\kek{3}\piros{1})) = \#\{\kek{3}\} -1 = 0$.

\end{example}

\begin{corollary}
	The above map defines a bijection $(\mathcal{C}_n^k, \wcRL) \to (\mathcal{E}_n^k, \welr)$, i.e., $\mathcal{C}_n^k (x; \wcRL) = \mathcal{E}_n^k (x; \welr)$.
\end{corollary}

By direct enumeration of $\mathcal{C}_n^k$ with the weight $\wcRL$, we can see that $\mathcal{C}_n^k(x; \wcRL)$ satisfies \cref{SPB-exp}. Thus, again, we have $\mathcal{E}_n^k (x; \welr) = \widehat{\mathscr{B}}_n^k(x)$.

\section{An application and remarks}\label{s5}

Now, we have various pairs of combinatorial models, $(\mathcal{C}_n^k, \wclr)$, $(\mathcal{T}_n^k, \wtst)$, $(\widetilde{\mathcal{T}}_n^k, \wttl)$, $(\mathscr{T}_n^k, \wtc)$, $(\mathcal{E}_n^k, \welr)$, etc. These models all provide the same polynomial $\widehat{\mathscr{B}}_n^k(x)$ and have their own characteristics. In this last section, we provide an application of the model $(\widetilde{\mathcal{T}}_n^k, \wttl)$.


\subsection{Combinatorial explanation of the duality}\label{s4-3}

Although the definition of the symmetrized poly-Bernoulli polynomial \cref{SPB-exp} clearly implies the duality $\widehat{\mathscr{B}}_n^k(x) = \widehat{\mathscr{B}}_k^n(x)$, it is unclear from the definitions of our combinatorial polynomials. To explain the duality combinatorially, we introduce another weight function for the packed alternative tableau.
For each $\lambda \in \widetilde{\mathcal{T}}_n^k$, the weight $\wttd(\lambda)$ counts the number of rows that contain $\downarrow$ and $\leftarrow$ in the leftmost column. 

\begin{theorem}
	For any $n, k \geq 0$, we have $\widetilde{\mathcal{T}}_n^k(x; \wttl) = \widetilde{\mathcal{T}}_n^k(x; \wttd)$.
\end{theorem}

\begin{proof}
We construct a bijection (involution) $f: (\widetilde{\mathcal{T}}_n^k, \wttl) \to (\widetilde{\mathcal{T}}_n^k, \wttd)$.
For $\lambda \in \widetilde{\mathcal{T}}_n^k$, we define $f(\lambda)$ as follows:
	\begin{itemize}
		\item[(1)] Consider all columns containing $\leftarrow$ and $\downarrow$ in the bottom row. For each column, we slide the $\downarrow$ to the location of the lowest $\leftarrow$ and slide the lowest $\leftarrow$ to the leftmost column.
		\item[(2)] Consider all rows containing $\downarrow$ and $\leftarrow$ in the leftmost column. For each row, we slide the $\leftarrow$ to the location of the most-left $\downarrow$ and slide the most-left $\downarrow$ to the bottom row.
	\end{itemize}
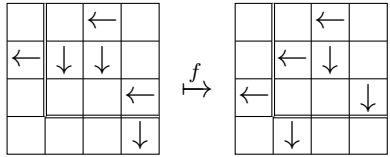
\begin{figure}[H]
	\begin{tikzpicture}
		\draw (-5/2,0)--(-1/2,0)--(-1/2,4/2)--(-5/2,4/2)--cycle;
		\draw (-5/2,1/2)--(-1/2,1/2);
		\draw[double] (-4/2,1/2)--(-1/2,1/2);
		\draw (-5/2,2/2)--(-1/2,2/2);
		\draw (-5/2,3/2)--(-1/2,3/2);
		\draw (-4/2,0)--(-4/2,4/2);
		\draw[double] (-4/2,1/2)--(-4/2,4/2);
		\draw (-3/2,0)--(-3/2,4/2);
		\draw (-2/2,0)--(-2/2,4/2);
		\draw (-5/2+1/4,2/2+1/4) node{$\leftarrow$};
		\draw (-2/2+1/4,1/2+1/4) node{$\leftarrow$};
		\draw (-3/2+1/4,3/2+1/4) node{$\leftarrow$};
		\draw (-3/2+1/4,2/2+1/4) node{$\downarrow$};
		\draw (-4/2+1/4,2/2+1/4) node{$\downarrow$};
		\draw (-2/2+1/4,0/2+1/4) node{$\downarrow$};
		
		\draw (0,1) node{$\xmapsto{f}$};
		
		\draw (3+-5/2,0)--(3-1/2,0)--(3-1/2,4/2)--(3-5/2,4/2)--cycle;
		\draw (3-5/2,1/2)--(3-1/2,1/2);
		\draw[double] (3-4/2,1/2)--(3-1/2,1/2);
		\draw (3-5/2,2/2)--(3-1/2,2/2);
		\draw (3-5/2,3/2)--(3-1/2,3/2);
		\draw (3-4/2,0)--(3-4/2,4/2);
		\draw[double] (3-4/2,1/2)--(3-4/2,4/2);
		\draw (3-3/2,0)--(3-3/2,4/2);
		\draw (3-2/2,0)--(3-2/2,4/2);
		\draw (1/2+1/4,1/2+1/4) node{$\leftarrow$};
		\draw (2/2+1/4,2/2+1/4) node{$\leftarrow$};
		\draw (3/2+1/4,3/2+1/4) node{$\leftarrow$};
		\draw (2/2+1/4,0/2+1/4) node{$\downarrow$};
		\draw (3/2+1/4,2/2+1/4) node{$\downarrow$};
		\draw (4/2+1/4,1/2+1/4) node{$\downarrow$};
	\end{tikzpicture}
	\caption{Example of the mapping.}
\end{figure}
The map is a well-defined involution.
Moreover, we have $\wttl(\lambda)=\wttd(f(\lambda))$ for all $\lambda \in \widetilde{\mathcal{T}}_n^k$.
\end{proof}

By reflecting packed alternative tableaux, we have $\widetilde{\mathcal{T}}_n^k(x; \wttl) = \widetilde{\mathcal{T}}_k^n(x; \wttd)$. Thus, we obtain the duality
\[
	\widetilde{\mathcal{T}}_n^k(x; \wttl) = \widetilde{\mathcal{T}}_k^n(s; \wttl).
\]
This provides a combinatorial interpretation of the duality formula.

\subsection{Concluding remarks}\label{s4-4}

Recently, B\'{e}nyi and the second author~\cite{BenyiMatsusaka2021-arXiv} and B\'{e}nyi--Ram\'{i}rez~\cite{BenyiRamirez2022} introduced combinatorial models for (non-symmetrized) poly-Bernoulli polynomials, poly-Euler numbers, and poly-Cauchy numbers based on the idea of Callan sequences. Using the interpretations, they provided combinatorial proofs for a large variety of known or new equations. It would be interesting to understand these polynomials and numbers using our various combinatorial models. As explained in this article, our combinatorial objects and weights have their own advantages. Do our models provide new aspects of these polynomials and numbers?

For each combinatorial set $\mathcal{P}$, there are many possibilities for weight functions. In our article, we defined three weights, $\wclr, \wcrb$, and $\wcRL$, for the set of double Callan permutations. As a similar phenomenon, for instance, Dumont--Foata~\cite{DumontFoata1976} introduced three weights for the set of (surjective) pistols. The corresponding three polynomials define the same polynomial, namely, the Gandhi polynomial. Furthermore, it is known that the Gandhi polynomial coincides with the anti-diagonal alternating sum of the symmetrized poly-Bernoulli polynomials~\cite{Matsusaka2020}. This result was shown indirectly by using the recurrence relations. Is it possible to provide its combinatorial proof? (See also B\'{e}nyi--Josuat-Verg\`{e}s~\cite{BenyiJosuatVerges2020}).


\bibliographystyle{amsplain}
\bibliography{References}

\end{document}